\documentclass{amsart}
\usepackage{amsmath,amsthm,amssymb,amscd}
\usepackage{wrapfig}

\newtheorem{thm}{Theorem}[section]
\newtheorem{prop}[thm]{Proposition}
\newtheorem{lem}[thm]{Lemma}
\newtheorem{cor}[thm]{Corollary}

\theoremstyle{definition}
\newtheorem{dfn}[thm]{Definition}
\theoremstyle{definition}
\newtheorem{eg}[thm]{Example}

\theoremstyle{remark}
\newtheorem{rem}[thm]{Remark}

\newcommand{\Hom}{{\rm Hom}}
\newcommand{\Map}{{\rm Map}}
\newcommand{\Sing}{{\rm Sing}}

\newcommand{\A}{\mathsf{A}}
\newcommand{\K}{\mathsf{K}}
\renewcommand{\L}{\mathsf{L}}
\newcommand{\Sd}{\mathsf{Sd}}

\newcommand{\Del}{\mathsf{\Delta}}
\newcommand{\Lam}{\mathsf{\Lambda}}

\newcommand{\Z}{\mathbb{Z}}

\newcommand{\D}{\Delta}

\title{Box complexes and homotopy theory of graphs}

\begin{document}

\author{Takahiro Matsushita}
\email{mtst@math.kyoto-u.ac.jp}
\address{Department of Mathematics, Kyoto University, Kitashirakawa Oiwake-cho Sakyo-ku, Kyoto 606-8502, Japan}


\begin{abstract}
We introduce a model structure on the category of graphs, which is Quillen equivalent to the category of $\Z_2$-spaces. A weak equivalence is a graph homomorphism which induces a $\Z_2$-homotopy equivalence between their box complexes. The box complex is a $\Z_2$-space associated to a graph, considered in the context of the graph coloring problem. In the proof, we discuss the universality problem of the Hom complex.
\end{abstract}

\maketitle

\section{Introduction}

We consider the category of graphs from the viewpoint of homotopical algebra. As a result, we construct a model structure on the category of graphs which is Quillen equivalent to the category of $\Z_2$-spaces. A weak equivalence is a graph homomorphism which induces a $\Z_2$-homotopy equivalence between their box complexes.

The box complex was introduced in the context of the graph coloring problem. An {\it $n$-coloring of a graph $G$} is a map from the vertex set of $G$ to the $n$-point set $\{ 1,\cdots, n\}$ so that adjacent vertices have different values. The {\it chromatic number $\chi(G)$ of $G$} is the smallest integer $n$ such that $G$ has an $n$-coloring. The graph coloring problem is to determine the chromatic number, and this is one of the most classical problems in graph theory.

The first application of algebraic topology to this subject is Lov\'asz's proof of the Kneser conjecture \cite{Lovasz}. Lov\'asz introduced the neighborhood complex $N(G)$ of a graph $G$, and showed that if the neighborhood complex of $G$ is $n$-connected, then the chromatic number of $G$ is greater than $n + 2$. The box complex $B(G)$ is a $\Z_2$-space which is homotopy equivalent to the neighborhood complex $N(G)$. The precise definition will be found in Section 2.

The Hom complex $\Hom(T,G)$ is a generalization of the box complex. If $T$ is a right $\Gamma$-graph, then the Hom complex $\Hom(T,G)$ becomes a left $\Gamma$-space and a graph homomorphism $f: G_1 \rightarrow G_2$ induces a $\Gamma$-map $f_* : \Hom(T,G_1) \rightarrow \Hom(T,G_2)$. Since an $n$-coloring of a graph $G$ is identified with a graph homomorphism from $G$ to $K_n$, we have that if there is no $\Gamma$-map from $\Hom(T,G)$ to $\Hom(T,K_n)$ then we have $\chi(G) \geq n$. The equivariant homotopy of the Hom complexes has been extensively researched (see \cite{BK1}, \cite{BK2}, \cite{DS}, \cite{Matsushita 1}, and \cite{Schultz} for example).

Therefore it is important to compare the category of graphs with the category of $\Gamma$-spaces of some group $\Gamma$. This is the motivation of this research. Our main result (Theorem \ref{thm 5.1}) asserts that the category of graphs has the model structure whose homotopy category is equivalent to the homotopy category of $\Z_2$-spaces.

\subsection{Singular complex functor and its adjoint}

Let $\Gamma$ be a finite group and $T$ a finite right $\Gamma$-graph. The functor $G \mapsto \Hom_T(G) = \Hom(T,G)$ has neither a left nor a right adjoint, and hence it is not a Quillen functor. So we use the singular complex functor $\Sing_T(G) = \Sing(T,G)$ introduced in \cite{Matsushita 2}. It is known that the singular complexes and the Hom complexes are homotopy equivalent (see Theorem \ref{thm sing_hom_eq} and Corollary \ref{cor 4.2.2}), and we will show that the functor
$$\Sing_T : \mathcal{G} \longrightarrow {\bf SSet}^{\Gamma}, \; G \mapsto \Sing(T,G)$$
is a right adjoint functor (Proposition \ref{prop 4.2.3}). Here $\mathcal{G}$ denotes the category of graphs and ${\bf SSet}^\Gamma$ is the category of $\Gamma$-simplicial sets. Let $A_T$ denote the left adjoint of $\Sing_T$. In the case of simplicial complexes, a similar construction $\A_T(\K)$ was obtained in \cite{DS} (see the end of Section 3.1).

Consider the unit of the adjoint pair
$$(A_T \circ {\rm Sd}^k, {\rm Ex}^k \circ {\rm Sing}_T) : {\bf SSet}^{\Gamma} \longrightarrow \mathcal{G},$$
where ${\rm Ex}$ is Kan's extension functor (see Section 2.2). We take $k$ to be sufficiently large. If this adjoint pair is a Quillen equivalence between ${\bf SSet}^{\Gamma}$ and $\mathcal{G}$, then the unit
$${\rm Id} \longrightarrow {\rm Ex}^k \circ {\rm Sing}_T \circ A_T \circ {\rm Sd}^k$$
must be a natural weak equivalence. The main task of Section 3 is to characterize the condition of $T$ that the unit is a natural $\Gamma$-weak equivalence.

For a right $\Gamma$-graph $T$ and for an element $\gamma$ of $\Gamma$, let $\alpha_\gamma$ denote the graph homomorphism $G \rightarrow G$, $v \mapsto v \gamma$. Consider the following two conditions concerning a finite right $\Gamma$-graph $T$:

\begin{itemize}
\item[(A)] For each subgroup $\Gamma'$ of $\Gamma$, the map $\Gamma / \Gamma' \rightarrow \Hom(T,T/\Gamma')$, $\gamma \Gamma' \mapsto p \circ \alpha_\gamma$ is a $\Gamma$-homotopy equivalence. The $\Gamma$-action on $\Hom(T,T/\Gamma')$ is described in Section 2.1.
\item[(B)] The map $\Gamma \mapsto \Hom(T,T)$, $\gamma \mapsto \alpha_\gamma$ is a $\Gamma$-homotopy equivalence.
\end{itemize}

Here we consider $\Gamma / \Gamma'$ as a discrete space and $p$ denotes the projection $T \rightarrow T/\Gamma'$. Clearly, (A) implies (B). The main structural results of Section 3 are the following two theorems:

\begin{thm}[Theorem \ref{thm 3.4.1}] \label{thm 1.1}
Let $\Gamma$ be a finite group and $T$ a finite connected right $\Gamma$-graph with at least one edge. Let $k$ be an integer such that $2^{k-2}$ is greater than the diameter of $T$. Then the following are equivalent:
\begin{itemize}
\item[(1)] The unit of the adjoint pair $(A_T \circ {\rm Sd}^k , {\rm Ex}^k \circ  {\rm Sing}_T)$ is a natural $\Gamma$-weak equivalence.
\item[(2)] The right $\Gamma$-graph $T$ satisfies the condition (A).
\end{itemize}
\end{thm}

\begin{thm}[Theorem \ref{thm 3.4.2}] \label{thm 1.2}
Let $\Gamma$ be a finite group and $T$ a finite connected right $\Gamma$-graph with at least one edge. Let $k$ be an integer such that $2^{k-2}$ is greater than the diameter of $T$. Then the following are equivalent:
\begin{itemize}
\item[(1)] For a free $\Gamma$-simplicial set $K$, the unit map $K \rightarrow {\rm Ex}^k \circ \Sing_T \circ A_T \circ {\rm Sd}^k(K)$ is a $\Gamma$-weak equivalence.
\item[(2)] The right $\Gamma$-graph $T$ satisfies the condition (B).
\end{itemize}
\end{thm}

If the graph $T$ is stiff (see Section 2.1 for the definition), then the condition (B) has the following combinatorial characterization: Every endomorphism of $T$ is an automorphism and the group $\Gamma$ is isomorphic to the automorphism group of $T$ (Lemma \ref{lem 3.4.3}). Such examples are given by complete graphs, odd cycles, and stable Kneser graphs (see Example \ref{eg 22}).

Theorem \ref{thm 1.2} concerns the universality problem of the Hom complexes. Csorba \cite{Csorba} showed that for every finite free $\Z_2$-complex $X$, there is a graph $G$ such that $\Hom(K_2,G)$ and $X$ are $\Z_2$-homotopy equivalent. Dochterman and Schultz \cite{DS} showed that for every free $S_n$-complex $X$, there is a graph $G$ such that $\Hom(K_n,G)$ and $X$ are $S_n$-homotopy equivalent. Here $S_n$ is the symmetric group of the $n$-element set $\{ 1,\cdots, n\}$. Thus Theorem \ref{thm 1.2} is a generalization of their results.

On the other hand, the condition (A) is rarely satisifed. I think that the following are (essntially) only examples naturally arising:
\begin{itemize}
\item[(1)] $\Gamma$ is trivial and $T$ is the graph ${\bf 1}$ consisting of one looped vertex. This corresponds to the clique complex of the maximal reflexive subgraph.
\item[(2)] $\Gamma$ is the cyclic group $\Z_2$ of order 2 and $T$ is $K_2$ with the flipping involution. This corresponds to the box complex.
\end{itemize}
This is a specific difference between the box complex and the other Hom complexes.

\subsection{Organization of the paper}

In Section 2, we introduce the notation and the terminology, and review some relevant facts. In Section 3, we state Theorem \ref{thm 1.1} and Theorem \ref{thm 1.2} and construct the model structures on the category of graphs (Theorem \ref{thm 5.1} and Theorem \ref{thm 5.2}). Here we give their proofs based on some structural results (Lemma \ref{lem 3.4.7} and Proposition \ref{prop 3.4.8}). These structural results are obtained in Section 4 by comparing the (equivariant) strong homotopy theory of simplicial complexes \cite{BM} with the $\times$-homotopy theory established by Dochtermann \cite{Dochtermann}.

\vspace{3mm} \noindent
{\bf Acknowledgements.}
I would like to express my gratitude to Toshitake Kohno for his indispensable advice and support. I would like to thank Dai Tamaki for his kind support during my stay at Shinshu University. I would like to thank Shouta Tounai for careful reading and helpful suggestions. Moreover, he provided me the proof of Proposition \ref{prop 4.2.3}. I thank Anton Dochtermann for valuable comments and a detailed answer to my question concerning his result. I also thank Jun Yoshida and an anonymous referee for valuable comments. This work was supported by a Grand-in-Aid for JSPS Fellows (No. ~25-4699 and 28-6304). This work was supported by the Program for Leading Graduate Schools, MEXT, Japan.

\section{Preliminaries}

In this section we shall introduce the notation and terminology, and review relevant facts which we will use in later sections.

Throughout the paper, $\Gamma$ denotes a finite group unless otherwise stated. We consider not only left group actions but also right group actions. However, we shall use the term ``a $\Gamma$-action" to mean ``a left $\Gamma$-action". For a category $\mathcal{C}$, we write $\mathcal{C}^{\Gamma}$ to indicate the category of objects in $\mathcal{C}$ equipped with a $\Gamma$-action.

For a poset $P$, the classifying space of $P$ is denoted by $|P|$. We often regard a poset as a topological space by its classifying space, and assign topological terminology by the classifying space functor. For example, two order-preserving maps $f, g: P \rightarrow Q$ are homotopic if and only if $|f|$ and $|g|$ are homotopic.

\subsection{Box complexes and Hom complexes of graphs}

For a concrete introduction to this subject, we refer to \cite{Kozlov book}.

A {\it graph} is a pair $G = (V(G),E(G))$ consisting of a set $V(G)$ together with a symmetric subset $E(G)$ of $V(G) \times V(G)$, i.e. $(x,y) \in E(G)$ implies $(y,x) \in E(G)$. A {\it graph homomorphism} is a map $f:V(G) \rightarrow V(H)$ such that $(f \times f)(E(G)) \subset E(H)$. Let $\mathcal{G}$ denote the category of graphs whose morphisms are graph homomorphisms.
For a vertex $v \in V(G)$, the {\it neighborhood $N_G(v)$ of $v$} is the set of vertices of $G$ adjacent to $v$. We sometimes abbreviate $N_G(v)$ to $N(v)$. Define the {\it complete graph $K_n$ with $n$-vertices} by $V(K_n) = \{ 1,\cdots, n\}$ and $E(K_n) = \{ (x,y) \; | \; x \neq y\}$. Then an $n$-coloring of $G$ is identified with a graph homomorphism from $G$ to $K_n$, and the chromatic number is
$$\chi(G) = \inf \{ n \geq 0 \; | \; \textrm{There is a graph homomorphism from $G$ to $K_n$.}\}.$$

The {\it box complex of a graph $G$} is the $\Z_2$-poset
$$B(G) = \{ (\sigma,\tau) \; | \; \sigma, \tau \in 2^{V(G)} \setminus \{ \emptyset\}, \sigma \times \tau \subset E(G)\}$$
ordered by the product of the inclusion orderings. The $\Z_2$-action of $B(G)$ is the exchange of the first and the second entries, i.e. $(\sigma ,\tau) \leftrightarrow (\tau,\sigma)$.

A {\it multi-homomorphism from $G$ to $H$} is a map $\eta : V(G) \rightarrow 2^{V(H)} \setminus \{ \emptyset\}$ such that $(v,w) \in E(G)$ implies $\eta(v) \times \eta(w) \subset E(H)$. For a pair of multi-homomorphisms $\eta$ and $\eta'$, we write $\eta \leq \eta'$ to mean that $\eta(v) \subset \eta'(v)$ for every vertex $v$ of $G$. The {\it Hom complex from $G$ to $H$} is the poset of the multi-homomorphisms from $G$ to $H$, and denoted by $\Hom(G,H)$.

This definition of the Hom complex is slightly different from the one of Babson and Kozlov \cite{BK1}. They define the Hom complex $\Hom(G,H)$ as a certain subcomplex of a direct product of simplices when $G$ and $H$ are finite. Our Hom complex is the face poset of theirs, and thus the topological types of these two definitions coincide.

Graph homomorphisms $f, g: G \rightarrow H$ are {\it $\times$-homotopic} (see \cite{Dochtermann}) if they belong to the same connected component of $\Hom(G,H)$. We write $f \simeq_\times g$ to mean that $f$ and $g$ are $\times$-homotopic. A graph homomorphism $f:G \rightarrow H$ is a {\it $\times$-homotopy equivalence} if there is a graph homomorphism $h : H \rightarrow G$ such that $hf \simeq_\times {\rm id}_G$ and $fh \simeq_\times {\rm id}_H$.

Let $G$, $H$, and $K$ be graphs. Define the composition map
$$* :\Hom(H,K) \times \Hom(G,H) \rightarrow \Hom(G,K), \; (\tau , \eta) \mapsto \tau * \eta$$
by
$$(\tau * \eta)(v) = \bigcup_{w \in \eta(v)} \tau(w).$$
If $f: G \rightarrow H$ and $g: H \rightarrow K$ are graph homomorphisms, we write $g_* (\eta)$ (or $f^*(\tau)$) instead of $g * \eta$ (or $\tau * f $, respectively).

Let $T$ be a right $\Gamma$-graph. For $\gamma \in \Gamma$, we write $\alpha_\gamma$ to indicate the graph homomorphism $T \rightarrow T$, $v \mapsto v \gamma$. Then we have a $\Gamma$-action on $\Hom(T,G)$ defined by $\gamma \eta = \alpha_\gamma^*(\eta)$, and a graph homomorphism $f: G_1 \rightarrow G_2$ induces a $\Gamma$-poset map $f_* : \Hom(T,G_1) \rightarrow \Hom(T,G_2)$. Consider $K_2$ as a $\Z_2$-graph by the flipping involution. Then it is clear that $\Hom(K_2, G)$ and $B(G)$ are isomorphic as $\Z_2$-posets.

\begin{lem}[See Theorem 5.1 of \cite{Dochtermann}] \label{lem 2.2.1}
Let $f$ and $g$ be graph homomorphisms from $G$ to $H$ and suppose $f \simeq_\times g$. For a right $\Gamma$-graph $T$, the following hold:
\begin{itemize}
\item[(1)] $f_*$, $g_* : \Hom(T,G) \rightarrow \Hom(T,H)$ are $\Gamma$-homotopic.
\item[(2)] $f^*$, $g^* : \Hom(H,T) \rightarrow \Hom(G,T)$ are $\Gamma$-homotopic.
\end{itemize}
\end{lem}
\begin{proof}
We only prove (1) since (2) is similarly proved. Let $\varphi : [0,1] \rightarrow |\Hom(G,H)|$ be a path joining $f$ to $g$. Then the composition of
$$\begin{CD}
[0,1] \times |\Hom(T,G)| @>{\varphi \times {\rm id}}>> |\Hom(G,H)| \times |\Hom(T,G)| @>{|*|}>> |\Hom(T,H)|
\end{CD}$$
is a $\Gamma$-homotopy from $f_*$ to $g_*$.
\end{proof}

A vertex $v$ of $G$ is {\it dismantlable} if there is another vertex $w$ of $G$ with $N(v) \subset N(w)$. If $v$ is dismantlable, then the inclusion $G \setminus v \hookrightarrow G$ is a $\times$-homotopy equivalence. Here $G \setminus v$ denotes the induced subgraph of $G$ whose vertex set is $V(G) \setminus \{ v\}$. In particular, we have the following:

\begin{cor}[Kozlov \cite{Kozlov}] \label{cor 2.2.2}
Let $G$ be a graph and $v$ a dismantlable vertex of $G$. For every right $\Gamma$-graph $T$, the inclusion $\Hom(T,G \setminus v) \hookrightarrow \Hom(T,G)$ is a $\Gamma$-homotopy equivalence.
\end{cor}

A graph $G$ is {\it stiff} if $G$ has no dismantlable vertices.

\begin{lem}[Lemma 6.5 of \cite{Dochtermann}] \label{lem stiff}
If $G$ is a stiff graph, then every automorphism of $G$ is an isolated point of $\Hom(G,G)$.
\end{lem}

\subsection{$\Gamma$-simplicial sets}

For a concrete explanation of simplicial sets, we refer to \cite{GJ}.

We write $\D$ to indicate the cosimplicial indexing category. Let ${\bf SSet}$ denote the category of simplicial sets. The geometric realization of a simplicial set $K$ is denoted by $|K|$. A simplicial map $f: K \rightarrow L$ is called a weak equivalence if the map $|f| : |K| \rightarrow |L|$ induced by $f$ is a homotopy equivalence.

For a $\Gamma$-simplicial set $K$ and a subgroup of $\Gamma'$, let $K^{\Gamma'}$ denote the subcomplex of $K$ consisting of the simplices fixed by $\Gamma'$. A $\Gamma$-simplicial map $f: K \rightarrow L$ is a $\Gamma$-weak equivalence if $f^{\Gamma'} : K^{\Gamma'} \rightarrow L^{\Gamma'}$ is a weak equivalence for every subgroup $\Gamma'$ of $\Gamma$.

Then ${\bf SSet}^\Gamma$ has the model structure described as follows (see  \cite{Bohamann et al} or \cite{Stephan}) whose generating cofibrations is
$$\mathcal{I}_\Gamma = \{ (\Gamma / \Gamma') \times \partial \D [n] \hookrightarrow (\Gamma / \Gamma') \times \D [n] \; | \; \textrm{$n \geq 0$ and $\Gamma' $ is a subgroup of $\Gamma $.}\}$$
and whose generating trivial cofibrations is
$$\mathcal{J}_\Gamma = \{ (\Gamma / \Gamma') \times \Lambda^r [n] \hookrightarrow (\Gamma / \Gamma') \times \D [n] \; | \; \textrm{$n \geq 1$, $0 \leq r \leq n$, and $\Gamma' $ is a subgroup of $\Gamma $.}\}.$$
Here $\D [n]$ is the Yoneda functor $[m] \mapsto \Delta ([m], [n])$ and $\Lambda^r[n]$ is the $r$-horn. The class of weak equivalences is the class of $\Gamma$-weak equivalences.

Let ${\rm Ex}$ denote Kan's extension functor, and ${\rm Sd}$ the barycentric subdivision functor. There is a natural weak equivalence ${\rm Sd}(K) \rightarrow K$ whose adjoint $K \rightarrow {\rm Ex}(K)$ is also a natural weak equivalence (Theorem 4.6 of \cite{GJ}). The adjoint pair $({\rm Sd}, {\rm Ex}) :{\bf SSet} \rightarrow {\bf SSet}$ gives rise to the adjoint pair $({\rm Sd}, {\rm Ex}) : {\bf SSet}^\Gamma \rightarrow {\bf SSet}^\Gamma$. It is easy to see that if $K$ is a $\Gamma$-simplicial sets, both of the above maps ${\rm Sd}(K) \rightarrow K$ and $K \rightarrow {\rm Ex}(K)$ are also $\Gamma$-weak equivalences.

\subsection{Bredon's theorem}

We will use the following proposition several times.

\begin{prop}[Bredon \cite{Bredon}] \label{Bredon}
Let $\Gamma$ be a finite group and $f: X \rightarrow Y$ a $\Gamma$-map between $\Gamma$-CW-complexes. Then $f$ is a $\Gamma$-homotopy equivalence if and only if $f^{\Gamma'} : X^{\Gamma'} \rightarrow Y^{\Gamma'}$ is a homotopy equivalence for every subgroup $\Gamma'$ of $\Gamma$.
\end{prop}

\subsection{Some homotopy colimits}

The following proposition is sometimes called the gluing lemma or cube lemma (Lemma 8.8 of \cite{GJ} or Proposition 15.10.10 of \cite{Hirschhorn}).

\begin{prop}\label{Cube lemma}
Let $\mathcal{C}$ be a model category and let
$$\begin{CD}
A @<i<< B @>>> C\\
@V{f_A}VV @VV{f_B}V @VV{f_C}V\\
A' @<{i'}<< B' @>>> C'
\end{CD}$$
be a commutative diagram in $\mathcal{C}$. Suppose that the all vertical arrows are weak equivalences, and the all objects appearing in the above diagram are cofibrant. If $i$ and $i'$ are cofibrations in $\mathcal{C}$, then the natural map
$$A \cup_B C \rightarrow A' \cup_{B'} C'$$
is a weak equivalence.
\end{prop}

For an ordinal $\lambda$, the minimum of $\lambda$ is denoted by $0$.

\begin{prop} \label{prop seq}
Let $\mathcal{C}$ be a model category, $X_\bullet : \lambda \rightarrow \mathcal{C}$, $Y_\bullet : \lambda \rightarrow \mathcal{C}$ functors from an ordinal $\lambda$, and $f_\bullet : X_\bullet \rightarrow Y_\bullet$ a natural transformation. Suppose the following conditions are satisfied:
\begin{itemize}
\item[(1)] $f_\alpha : X_\alpha \rightarrow Y_\alpha$ is a weak equivalence for every $\alpha < \lambda$.
\item[(2)] For each $\alpha < \lambda$, the map
$${\rm colim}_{\beta < \alpha} X_\beta \longrightarrow X_\alpha$$
is a cofibration of $\mathcal{C}$. In particular, $X_0$ is cofibrant.
\end{itemize}
Then the colimit $f_{\lambda} : X_{\lambda} \rightarrow Y_{\lambda}$ of $f_{\bullet}$ is a weak equivalence.
\end{prop}
\begin{proof}
The proof is similar to Proposition 15.10.12 of \cite{Hirschhorn}.
\end{proof}

\section{Simplicial methods}

Let $\Gamma$ be a finite group and $T$ a right $\Gamma$-graph. Let $\mathcal{P}$ denote the category of posets. As was mentioned in Section 1, the functor
$$\Hom_T : \mathcal{G} \longrightarrow \mathcal{P}^{\Gamma}, G \mapsto \Hom(T,G)$$
is neither a left nor right adjoint functor. So we use the singular complex functor $\Sing_T : \mathcal{G} \rightarrow {\bf SSet}^{\Gamma}$, which is reviewed in Section 3.1. This is a right adjoint functor (Proposition \ref{prop 4.2.3}) and let $A_T$ denote the left adjoint. Then we have an adjoint pair
$$(A_T \circ {\rm Sd}^k, {\rm Ex}^k \circ \Sing_T) : {\bf SSet}^\Gamma \longrightarrow \mathcal{G}$$
for $k > 0$. In Section 3.2, we characterize the condition that the unit
$${\rm Id} \longrightarrow {\rm Ex}^k \circ \Sing_T \circ A_T \circ {\rm Sd}^k$$
of the adjunction is a natural $\Gamma$-weak equivalences for sufficiently large $k$ (Theorem \ref{thm 3.4.1}). In this section we give the proof of Theorem \ref{thm 3.4.1} based on some structural results proved in Section 4. In Section 3.3, we construct a model structure on $\mathcal{G}$ which is Quillen equivalent to ${\bf SSet}^{\Z_2}$ (Theorem \ref{thm 5.1}).

\subsection{Singular complexes}

For a non-negative integer $n$, define the graph $\Sigma^n$ by $V(\Sigma^n) = [n]$ and $E(\Sigma^n) = V(\Sigma^n) \times V(\Sigma^n)$. For a pair of graphs $T$ and $G$, the {\it singular complex} (see \cite{Matsushita 2}) is the simplicial set ${\rm Sing}(T,G)$ whose $n$-simplices are the graph homomorphisms from $T \times \Sigma^n$ to $G$, i.e. $\Sing(T,G)_n = \mathcal{G} (T \times \Sigma^n, G)$. The face and degeneracy maps are defined in an obvious way. A $0$-simplex of $\Sing(T,G)$ is identified with a graph homomorphism from $T$ to $G$. The fundamental result of the singular complex is the following:

\begin{thm}[Matsushita \cite{Matsushita 2}] \label{thm sing_hom_eq}
There is a homotopy equivalence
$$\Phi : |\Sing(T,G)| \longrightarrow |\Hom(T,G)|,$$
which is natural with respect to both $T$ and $G$. Moreover, for a graph homomorphism $f: T \rightarrow G$, we have $\Phi(f) = f$.
\end{thm}

\begin{cor} \label{cor 4.2.2}
Let $\Gamma$ be a finite group and $T$ a right $\Gamma$-graph. Then the map
$$\Phi : |\Sing(T,G)| \longrightarrow |\Hom(T,G)|$$
in Theorem \ref{thm sing_hom_eq} is a $\Gamma$-homotopy equivalence.
\end{cor}
\begin{proof}
The naturality with respect to $T$ implies that $\Phi : |\Sing(T,G)| \rightarrow |\Hom(T,G)|$ is $\Gamma$-equivariant. It suffices to show that, for every subgroup $\Gamma'$ of $\Gamma$, the map
$$\Phi^{\Gamma'} : |\Sing(T,G)|^{\Gamma'} \longrightarrow |\Hom(T,G)|^{\Gamma'}$$
is a homotopy equivalence (see Proposition \ref{Bredon}).

Since the geometric realization preserves equalizers, we have
$$|\Sing(T,G)|^{\Gamma'} \cong |\Sing(T,G)^{\Gamma'}|, \; |\Hom(T,G)|^{\Gamma'} \cong |\Hom(T,G)^{\Gamma'}|.$$
Let $p : T \rightarrow T/ \Gamma'$ be the quotient map. Clearly, the maps
$$p^* : \Sing(T/\Gamma', G) \longrightarrow \Sing(T,G),$$
$$p^* : \Hom(T/ \Gamma' ,G) \longrightarrow \Hom(T,G)$$
are monomorphisms and their images coincide with $\Sing(T,G)^{\Gamma'}$ and $\Hom(T,G)^{\Gamma'}$, respectively. Since $\Phi : |\Sing(T/\Gamma', G)| \rightarrow |\Hom(T/\Gamma', G)|$ is a homotopy equivalence (Theorem \ref{thm sing_hom_eq}), this completes the proof.
\end{proof}

\begin{prop} \label{prop 4.2.3}
Let $\Gamma$ be a group and $T$ a right $\Gamma$-graph. Then the functor
$${\rm Sing}_T : \mathcal{G} \longrightarrow {\bf SSet}^{\Gamma}, \; G \longmapsto \Sing(T,G)$$
has a left adjoint.
\end{prop}

I am pleased to mention that Shouta Tounai provides me the following sophisticated proof.

\vspace{2mm} \noindent {\it Proof of Proposition \ref{prop 4.2.3}.} Let $\Gamma$ be a group and consider $\Gamma$ as a small category in the usual way. Namely, the object set is the one point set $\{ *\}$, the set of endomorphisms of $*$ is $\Gamma$, and the composition is the multiplication of $\Gamma$.

A right $\Gamma$-graph is identified with a functor from $\Gamma^{\rm op}$ to $\mathcal{G}$. Thus we have a functor from $\Gamma^{\rm op} \times \Delta \rightarrow \mathcal{G}$, $[n] \mapsto T \times \Sigma^n$. The associated functor
$$\mathcal{G}(T \times \Sigma^{\bullet}, -) : \mathcal{G} \longrightarrow {\bf Set}_{\Gamma^{\rm op} \times \Delta} ={\bf Set}^{\Gamma \times \Delta^{\rm op}} \cong {\bf SSet}^{\Gamma}$$
coincides with $\Sing_T$. Thus the left Kan extension of $T \times \Sigma^\bullet$ along the Yoneda functor $\Gamma^{\rm op} \times \Delta \rightarrow {\bf Set}$ is the left adjoint of $\Sing_T$.
\qed

\vspace{2mm}
Let $T$ be a right $\Gamma$-graph. Let $A_T$ denote the left adjoint of the functor $\Sing_T : \mathcal{G} \longrightarrow {\bf SSet}^{\Gamma}$. We shall precisely describe the graph $A_T(K)$ for a $\Gamma$-simplicial set $K$. First we construct a $\Gamma$-graph $A(K)$. The vertex set of $A(K)$ is the set $K_0$ of $0$-simplices, and two $0$-simplices $v$ and $w$ of $K$ are adjacent in $A(K)$ if and only if there is a 1-simplex connecting them. The group $\Gamma$ acts on $T \times A(K)$ by $\gamma (x,v) = (x \gamma^{-1}, \gamma v)$. Then the graph $A_T(K)$ is the quotient graph $\Gamma \backslash (T \times A(K))$.

In the case of simplicial complexes, the following similar construction $\A_T(\K)$ was known: We first construct a $\Gamma$-graph $\A(\K)$ for a $\Gamma$-simplicial complex $\K$. The vertex set of $\A(\K)$ is the vertex set $V(\K)$ of $\K$. Two vertices $v$ and $w$ of $\K$ are adjacent in $\A(\K)$ if and only if the set $\{ v,w\}$ is a simplex of $\K$. Then the graph $\A_T(\K)$ is the quotient $\A_T(\K) = \Gamma \backslash (T \times \A(\K))$. This construction was first considered by Csorba \cite{Csorba} in case $T$ is $K_2$, and was later generalized in Dochtermann and Schultz \cite{DS}.

Let $K$ be a simplicial set whose cell structure is isomorphic to some simplicial complex $\K$. Then it is clear that the graphs $A_T(K)$ and $\A_T(\K)$ are isomorphic.

\subsection{Unit of the adjoint pair}

{\it Throughout this section, $T$ is a finite connected right $\Gamma$-graph with at least one edge.} In this section, we characterize the condition that the unit
\begin{eqnarray} \label{eqn 4.4.1}
{\rm Id} \longrightarrow {\rm Ex}^k \circ \Sing_T \circ A_T \circ {\rm Sd}^k
\end{eqnarray}
is a natural $\Gamma$-weak equivalence for sufficiently large $k$ (Theorem \ref{thm 3.4.1}).

First we consider the following two conditions concerning a finite right $\Gamma$-graph $T$. Here we consider a $\Gamma$-set $X$ as the simplicial set whose 0-simplices are the elements of $X$ and which has no other non-degenerate simplices. Recall that for an element $\gamma$ of $\Gamma$, the graph homomorphism $\alpha_\gamma : T \rightarrow T$ is defined by $x \mapsto x \gamma$.

\begin{itemize}
\item[(A)] For each subgroup $\Gamma'$ of $\Gamma$, the map
$$\Gamma / \Gamma'  \longrightarrow \Sing(T,T / \Gamma'), \; \gamma \Gamma' \longmapsto p \circ \alpha_\gamma$$
is a $\Gamma$-weak equivalence.
\item[(B)] The map
$$\Gamma \longrightarrow \Sing(T,T), \; \gamma \mapsto \alpha_\gamma$$
is a $\Gamma$-weak equivalence.
\end{itemize}
The condition (A) implies the condition (B). One can show that the map
$$\Gamma / \Gamma' \longrightarrow \Sing(T, T /\Gamma') = \Sing_T \circ A_T (\Gamma /\Gamma')$$
in the condition (A) is the unit map of $(A_T, \Sing_T)$. Thus if the unit (\ref{eqn 4.4.1}) is a natural $\Gamma$-weak equivalence for every $\Gamma$-simplicial set $K$, then the condition (A) holds (see also (4) of Section 2.3). On the other hand, the following result asserts that the converse also holds:

\begin{thm} \label{thm 3.4.1}
Let $\Gamma$ be a finite group and $T$ a finite connected right $\Gamma$-graph having at least one edge and diameter $r$. Let $k$ be a positive integer such that $2^{k-2} > r$. If $T$ satisfies the condition (A), then the unit map
$$u_K : K \longrightarrow {\rm Ex}^k \circ \Sing_T \circ A_T \circ {\rm Sd}^k (K)$$
is a $\Gamma$-weak equivalence for every $\Gamma$-simplicial set $K$.
\end{thm}

If we restrict our attention to free $\Gamma$-simplicial sets, the following holds:

\begin{thm} \label{thm 3.4.2}
Let $\Gamma$ be a finite group and $T$ a finite connected right $\Gamma$-graph having at least one edge and diameter $r$. Let $k$ be an integer with $2^{k-2} > r$. If $T$ satisfies the condition (B), then the unit map
$$u_K : K \longrightarrow {\rm Ex}^k \circ \Sing_T \circ A_T \circ {\rm Sd}^k(K)$$
is a $\Gamma$-weak equivalence for every free $\Gamma$-simplicial set $K$.
\end{thm}

Before giving the proofs, we consider when the right $\Gamma$-graph $T$ satisfies the above conditions. In fact the condition (A) is a quite strong requirement. I think that the following are the only examples naturally arising.

\begin{itemize}
\item[(1)] $\Gamma$ is trivial and $T$ is the graph ${\bf 1}$ consisting of one looped vertex.
\item[(2)] $\Gamma$ is $\Z_2$ and $T$ is $K_2$ with the $\Z_2$-action which flips the edge of $K_2$.
\end{itemize}
Of course, if $T$ is $\times$-homotopy equivalent to ${\bf 1}$ or $\Z_2$-$\times$-homotopy equivalent to $K_2$, then $T$ satisfies the condition (A). Note that $\Hom({\bf 1}, G)$ is the face poset of the clique complex of the maximal reflexive subgraph of $G$, i.e. the induced subgraph of $G$ whose vertices are looped vertices of $G$. On the other hand, the case (2) corresponds to the box complex $B(G) = \Hom(K_2,G)$.

On the other hand, there are several graphs such that the condition (B) is satisfied. In case the graph $T$ is stiff (see Section 2.1), the condition (B) has the following combinatorial characterization. Here ${\rm Aut}(T)$ denotes the automorphism group of $T$.

\begin{lem} \label{lem 3.4.3}
Suppose that $T$ is stiff. Then $T$ satisfies the condition (B) if and only if the homomorphism $\alpha : \Gamma \rightarrow {\rm Aut}(T)$, $\gamma \mapsto \alpha_\gamma$ is an isomorphism and every endomorphism of $T$ is an automorphism.
\end{lem}
\begin{proof}
Recall that the composition of
\begin{eqnarray}\label{eqn 4.15.1}
\Gamma \longrightarrow {\rm Aut}(T) \longrightarrow {\rm Sing}(T,T)
\end{eqnarray}
coincides with the unit map. Suppose that the homomorphism $\Gamma \rightarrow {\rm Aut}(T)$ is an isomorphism and every endomorphism of $T$ is an automorphism. Lemma \ref{lem stiff} implies ${\rm Aut}(T) \cong \Hom(T,T)$ and hence the unit $\Gamma \rightarrow \Sing(T,T) \simeq \Hom(T,T)$ is a $\Gamma$-weak equivalence. The proof of the converse is similar and is omitted.
\end{proof}

\begin{eg}\label{eg 22}
There are several examples which satisfy the condition (B):

\begin{itemize}
\item[(1)] $\Gamma$ is the symmetric group $S_n$ of the $n$-element set $\{ 1,\cdots, n\}$, and $T$ is the complete graph $K_n$ for $n \geq 2$.
\item[(2)] $\Gamma$ is the dihedral group $D_{2n+1}$ of order $(4n+2)$, and $T$ is the odd cycle $C_{2n+1}$ with length ${2n+1}$ for $n \geq 1$.
\item[(3)] A subset $\sigma$ of the cyclic group $\Z_n$ of order $n$ is {\it stable} if $x \in \sigma$ implies $x+1 \not\in \sigma$. The {\it stable Kneser graph $SG_{n,k}$} is the graph whose vertex set consists of the stable $k$-subsets of $\Z_n$ and two stable $k$-subsets are adjacent if and only if they are disjoint. The stable Kneser graphs were introduced by Schrijver \cite{Schrijver}, and he showed that stable Kneser graphs are vertex critical, i.e. every subgraph $G$ of $SG_{n,k}$ such that $V(G) \subsetneq V(SG_{n,k})$ has a chromatic number smaller than $\chi(SG_{n,k})$. It is easy to see that the vertex critical finite graph satisfies the condition of Lemma \ref{lem 3.4.3}. Braun \cite{Braun} showed that the group of automorphisms of $SG_{n,k}$ is isomorphic to the dihedral group $D_n$ of order $2n$.
\end{itemize}
\end{eg}

Csorba \cite{Csorba} showed that for every $\Z_2$-CW-complex $X$, there is a simple graph $G$ such that $\Hom(K_2,G)$ and $X$ are $\Z_2$-homotopy equivalent. Theorem \ref{thm 3.4.1} implies that the free assumption is redundant:

\begin{cor} \label{cor 3.4.4}
For every $\Z_2$-CW-complex $X$, there is a graph $G$ such that $X$ and $\Hom(K_2,G)$ are $\Z_2$-homotopy equivalent.
\end{cor}

Let $S_n$ be the symmetric group of the $n$-element set $\{ 1,\cdots, n\}$. Then $S_n$ acts on $K_n$ in an obvious way. Dochtermann and Schultz \cite{DS} proved that for every free $S_n$-CW-complex $X$, there exists a graph $G$ such that $\Hom(K_n,G)$ and $X$ are $S_n$-homotopy equivalent \cite{DS}. The following is a generalization of \cite{DS}.

\begin{cor} \label{cor 3.4.5}
Suppose that a finite connected right $\Gamma$-graph $T$ with at least one edge satisfies the condition (B). Then for every free $\Gamma$-complex $X$, there is a graph $G$ such that $\Hom(T,G)$ and $X$ are $\Gamma$-homotopy equivalent.
\end{cor}

Here we need the assumption that the group action is free. In fact, $\Hom(K_n,G)$ is free if and only if $G$ has no looped vertices. On the other hand, if $G$ has looped vertices, then $\Hom(K_n,G)$ has fixed points.

The converse of Corollary \ref{cor 3.4.5} is false. In fact Dochtermann \cite{Dochtermann 2} showed that for every finite connected graph $T$ with at least one edge and for every CW-complex $X$, there is a graph $G$ such that $\Hom(T,G)$ and $X$ are homotopy equivalent.

Now we turn to the proofs of Theorem \ref{thm 3.4.1} and Theorem \ref{thm 3.4.2}. Let $\Del^n$ to indicate the simplicial complex $([n], 2^{[n]})$. For $n > 0$ and $0 \leq r \leq n$, define $\Lam_r^n$ to be the simplicial complex whose vertex set is $[n]$ and whose simplex is a subset $\sigma$ with $\sigma \cup \{ r\} \neq [n]$. Clearly, we have
$$A_T(\D [n]) = \A_T(\Del^n), \; A_T(\Lambda_r[n]) = \A_T(\Lam_r^n).$$

Let $k$ be an integer such that $2^{k-2}$ is greater than the diameter of $T$. In the rest of this section, we write $\hat{A}_T$ instead of $A_T \circ {\rm Sd}^k$ and $\hat{S}_T$ instead of ${\rm Ex}^k \circ {\rm Sing}_T$. For a subgroup $\Gamma'$ of $\Gamma$ and a simplicial set $K$, we write $K_{\Gamma'}$ instead of $(\Gamma / \Gamma') \times K$.

Now we turn to the proof of Theorem \ref{thm 3.4.1}. The proof is based on the following two assertions which will be proved in Section 4.

\begin{lem} \label{lem 3.4.7}
Let $\Gamma'$ be a subgroup of the finite group $\Gamma$, and suppose that $T$ satisfies the condition (A). Then the following hold:
\begin{itemize}
\item[(1)] The inclusion $\hat{A}_T(\Gamma / \Gamma') \hookrightarrow \hat{A}_T((\Gamma / \Gamma') \times \Lambda_r[n])$ induced by the inclusion $\Gamma /\Gamma' \hookrightarrow (\Gamma / \Gamma') \times \Lambda_r[n]$, $\gamma \mapsto (\gamma, r )$ is a $\times$-homotopy equivalence.
\item[(2)] The inclusion $\hat{A}_T(\Gamma / \Gamma') \hookrightarrow \hat{A}_T((\Gamma / \Gamma') \times \Delta [n])$ is a $\times$-homotopy equivalence.
\item[(3)] The unit map $\Lambda_r[n]_{\Gamma'} \longrightarrow \hat{S}_T \circ \hat{A}_T (\Lambda_r[n])_{\Gamma'}$ is a $\Gamma$-weak equivalence.
\item[(4)] The unit map $\Delta [n]_{\Gamma'} \longrightarrow \hat{S}_T \circ \hat{A}_T (\D[n]_{\Gamma'})$ is a $\Gamma$-weak equivalence.
\end{itemize}
\end{lem}

For a $\Gamma$-simplicial set $K$, we write $u_K$ to indicate the unit $K \rightarrow \hat{S}_T \circ \hat{A}_T (K)$.

\begin{prop}\label{prop 3.4.8}
Let $(K,L)$ be a pair of $\Gamma$-simplicial sets, $f : L \rightarrow L'$ a $\Gamma$-simplicial map, and let $K' = K \cup_L L'$. If the unit maps $u_L, u_K, u_{L'}$ are $\Gamma$-weak equivalences, then the unit map $u_{K'}$ is a $\Gamma$-weak equivalence.
\end{prop}

\begin{cor} \label{cor 3.4.9}
Suppose that $T$ satisfies the condition (A). For $n \geq 0$, the unit map $u_{\partial \Delta[n]} : (\Gamma / \Gamma') \times \partial \Delta[n] \rightarrow \hat{S}_T \circ \hat{A}_T((\Gamma / \Gamma') \times \partial \Delta[n])$ is a $\Gamma$-weak equivalence for every subgroup $\Gamma'$.
\end{cor}
\begin{proof}
We show this by the induction on $n$. The case $n = 0$ is obvious. Suppose $n > 0$. Regard $\Delta[n-1]$ as a subcomplex of $\Delta[n]$ by the map induced by the inclusion $[n-1] \hookrightarrow [n]$. Apply Proposition \ref{prop 3.4.8} to the case of $K = \Lambda^n[n]$, $L = \partial \Delta [n-1]$, and $L' = \Delta [n-1]$.
\end{proof}

\noindent {\it Proof of Theorem \ref{thm 3.4.1}.} Let $K$ be a $\Gamma$-simplicial set. Recall that a generating cofibrations $\mathcal{I}_\Gamma$ of ${\bf SSet}^{\Gamma}$ (Section 2.2) is described as follows:
$$\mathcal{I}_\Gamma = \{ (\Gamma / \Gamma') \times \partial \D [n] \hookrightarrow (\Gamma / \Gamma') \times \D [n] \; | \; \textrm{$n \geq 0$ and $\Gamma' $ is a subgroup of $\Gamma $.}\}$$
Since every $\Gamma$-simplicial set is an $\mathcal{I}_\Gamma$-cell complex, there exist an ordinal $\lambda$ and a colimit preserving functor $X_\bullet : \lambda \rightarrow {\bf SSet}^{\Gamma}$ such that $X_0 = \emptyset$, the colimit $X_\lambda$ of $X_\bullet$ is isomorphic to $K$, and $X_\alpha \rightarrow X_{\alpha + 1}$ is a pushout of an element of $\mathcal{I}_\Gamma$ for every $\alpha < \lambda$. We want to show that the unit map $X_\lambda \rightarrow \hat{S}_T \circ \hat{A}_T(X_\lambda)$ is a $\Gamma$-weak equivalence by the transfinite induction on $\alpha_0 < \lambda$. Suppose that for every $\alpha < \alpha_0$, the unit map $X_\alpha \rightarrow \hat{S}_T \circ \hat{A}_T(X_\alpha)$ is a $\Gamma$-weak equivalence. If $\alpha_0 - 1$ exists, then it follows from Lemma \ref{lem 3.4.7}, Proposition \ref{prop 3.4.8}, and Corollary \ref{cor 3.4.9} that the unit $X_{\alpha_0} \rightarrow \hat{S}_T \circ \hat{A}_T(X_{\alpha_0})$ is a $\Gamma$-weak equivalence. If $\alpha_0$ is a limit ordinal, then Proposition \ref{prop seq} implies that $X_{\alpha_0} \rightarrow \hat{S}_T \circ \hat{A}_T (X_{\alpha_0})$ is a $\Gamma$-weak equivalence since $\Sing_T$ and ${\rm Ex}$ preserve sequential colimits. This completes the proof of Theorem \ref{thm 3.4.1}. \qed

\vspace{2mm} We can prove Theorem \ref{thm 3.4.2} in a similar way. Define the small family $\mathcal{I}'_\Gamma$ of $\Gamma$-simplicial maps as follows:
$$\mathcal{I}'_\Gamma = \{ \Gamma \times \partial \Delta[n] \hookrightarrow \Gamma \times \Delta[n] \; | \; n \geq 0\}.$$
Then a $\Gamma$-simplicial set $K$ is free if and only if $K$ is an $\mathcal{I}'_\Gamma$-cell complex. Thus, in the proof of Theorem \ref{thm 3.4.1}, replacing ``condition (A)" to ``condition (B)" and considering only the trivial subgroup $1$, we have the proof of Theorem \ref{thm 3.4.2}.

\subsection{Model structure}

In this section, we introduce two model structures on the category of graphs.

\begin{thm} \label{thm 5.1}
The category $\mathcal{G}$ of graphs has the cofibrantly generated model structure with generating cofibrations $A_{K_2} \circ {\rm Sd}^3 (\mathcal{I}_{\Z_2})$ and with generating trivial cofibrations $A_{K_2} \circ {\rm Sd}^3 (\mathcal{J}_{\Z_2})$ (see Section 2.2). A graph homomorphism $f: G \rightarrow H$ is a weak equivalence if and only if the map $f_* : B(G) \rightarrow B(H)$ induced by $f$ is a $\Z_2$-homotopy equivalence. Moreover, the adjoint pair
$$\begin{CD}
(A_{K_2} \circ {\rm Sd}^3, {\rm Ex}^3 \circ \Sing_{K_2}) : {\bf SSet}^{\Z_2} @>>> \mathcal{G}
\end{CD}$$
is a Quillen equivalence.
\end{thm}

In the case of $\Sing_{\bf 1}$, we have the following theorem. Note that the Hom complex $\Hom ({\bf 1}, G)$ is the face poset of the clique complex of the maximal reflexive subgraph of $G$. We write $\mathcal{I}$ and $\mathcal{J}$ for $\mathcal{I}_{\Gamma}$ and $\mathcal{J}_{\Gamma}$ if the group $\Gamma$ is the trivial group $1$.

\begin{thm} \label{thm 5.2}
The category $\mathcal{G}$ of graphs has the cofibrantly generated model structure with generating cofibrations $A_{\bf 1} \circ {\rm Sd}^2(\mathcal{I})$ and with generating trivial cofibrations $A_{\bf 1} \circ {\rm Sd}^2 (\mathcal{J})$. A graph homomorphism $f: G \rightarrow H$ is a weak equivalence if and only if $f_* : \Hom({\bf 1} , G) \rightarrow \Hom({\bf 1}, H)$ induced by $f$ is a homotopy equivalence. Moreover, the adjoint pair
$$\begin{CD}
(A_{\bf 1} \circ {\rm Sd}^2 , {\rm Ex}^2 \circ \Sing_{\bf 1}) : {\bf SSet} @>>> \mathcal{G}
\end{CD}$$
is a Quillen equivalence.
\end{thm}

We only give the proof of Theorem \ref{thm 5.1} since the other is similar.

First we show that $\mathcal{G}$ has the model structure described in Theorem \ref{thm 5.1}. It is clear that every object of $\mathcal{G}$ is a small object in the sense of Definition 10.4.1 of \cite{Hirschhorn}. Thus by Theorem 11.3.2 of \cite{Hirschhorn}, it suffices to show that ${\rm Ex}^3 \circ \Sing_{K_2}$ takes a pushout of an element of $A_{K_2} \circ {\rm Sd}^3(\mathcal{J}_{\Z_2})$ to a $\Gamma$-weak equivalence. But this follows from (3) of Lemma \ref{lem 3.4.7} and Proposition \ref{prop 3.4.8}. Thus $\mathcal{G}$ has the model structure described in Theorem \ref{thm 5.1}.

Next we show that the adjoint pair $(\hat{A}_{K_2}, \hat{S}_{K_2}) \: =(A_{K_2} \circ {\rm Sd}^3, {\rm Ex}^3 \circ \Sing_{K_2})$ is a Quillen equivalence. By Corollary 1.3.16 of \cite{Hovey}, it suffices to verify the following:
\begin{itemize}
\item[(1)] Let $f: X \rightarrow Y$ be a graph homomorphism between fibrant objects in $\mathcal{G}$. If $\hat{S}_{K_2}(f)$ is a $\Z_2$-weak equivalence, then $f$ is a weak equivalence.
\item[(2)] For every $\Z_2$-simplicial set $K$, the composition of
$$K \rightarrow  \hat{S}_{K_2} \hat{A}_{K_2}(K) \rightarrow \hat{S}_{K_2} R \hat{A}_{K_2}(K)$$
is a $\Gamma$-weak equivalence. Here $R$ denotes a fibrant replacement functor of ${\bf SSet}^{\Gamma}$.
\end{itemize}
The definition of weak equivalences of $\mathcal{G}$ follows (1) and that the right arrow in (2) is a $\Gamma$-weak equivalence. Thus (2) follows from Theorem \ref{thm 3.4.1} since $K_2$ satisfies the condition (A) in Section 3.2. This completes the proof.

\section{Strong homotopy theory}

The purpose of this section is to show Lemma \ref{lem 3.4.7} and Proposition \ref{prop 3.4.8}. Lemma \ref{lem 3.4.7} is proved in Section 4.4 and Proposition \ref{prop 3.4.8} is proved in Section 4.6.

The difficulty of the proofs seems to lie in the following fact: Let $K$ be a $\Gamma$-simplicial set and $L$ a $\Gamma$-subcomplex of it. In general, $|\Sing_T \circ A_T(L)|$ is not a deformation retract of $|\Sing_T \circ A_T (K)|$ even if $|L|$ is a deformation retract of $|K|$. On the other hand, the strong collapses of $\Gamma$-complexes and $\times$-homotopy deformation retracts, which will be introduced later, have the following important properties:
\begin{itemize}
\item If a graph $H$ is a $\times$-deformation retract of $G$, then $|\Sing_T(H)|$ is a deformation retract of $|\Sing_T(G)|$ (Lemma \ref{lem 2.2.1}).
\item If a $\Gamma$-simplicial complex $\K$ strongly $\Gamma$-collapses to its $\Gamma$-subomplex $\L$, then $\A_T(\L)$ is a $\times$-deformation retract of $\A_T(\K)$ (Proposition \ref{prop 4.3.1}).
\item For any pair of $(\K , \L)$ of finite $\Gamma$-simplicial complexes, $\Sd^k (\L)$ has a large neighborhood in $\Sd^k(\K)$ which strongly $\Gamma$-collapses to $\Sd^k(\L)$ if we take $k$ to be sufficiently large (Corollary \ref{cor 2.4.9}).
\end{itemize}

\subsection{Simplicial complexes}

An {\it (abstract) simplicial complex} consists of a family $\K$ of finite subsets of a set $S$ such that $\sigma \in \K$ and $\tau \subset \sigma$ imply $\tau \in \K$. The {\it vertex set $V(\K)$ of $\K$} is the union of the all simplices of $\K$. Let $\K$ and $\L$ be simplicial complexes. A {\it simplicial map from $\K$ to $\L$} is a map $f: V(\K) \rightarrow V(\L)$ such that $\sigma \in \K$ implies $f(\sigma) \in \L$. The {\it geometric realization of $\K$} (see \cite{Kozlov book}) is denoted by $|\K |$. We assign the topological terminology to simplicial complexes by the geometric realization functor. For example, a simplicial map $f: \K \rightarrow \L$ is a homotopy equivalence if and only if the continuous map $|f| : |\K | \rightarrow |\L |$ induced by $f$ is a homotopy equivalence.

Let $\K$ and $\L$ be simplicial complexes. A map $\eta : V(\K) \rightarrow 2^{V(\L)} \setminus \{ \emptyset\}$ is a {\it simplicial multi-map} if, for every simplex $\sigma$ of $\K$, the subset $\bigcup_{v \in \sigma} \eta(v)$ of $V(\L)$ is a simplex of $\L$. For a pair of simplicial multi-maps $\eta$ and $\eta'$, we write $\eta \leq \eta'$ to mean that $\eta(v) \subset \eta'(v)$ for every $v \in V(\K)$. The poset of simplicial multi-maps from $\K$ to $\L$ is denoted by ${\rm Map}(\K,\L)$. A simplicial map is identified with a minimal point of $\Map(\K,\L)$. Two simplicial maps $f$ and $g$ are {\it strongly homotopic} if they belong to the same connected component of $\Map(\K,\L)$.

Recall that two simplicial maps $f, g : \K \rightarrow \L$ are {\it contiguous} if $\sigma \in \K$ implies $f(\sigma) \cup g(\sigma) \in \L$. Hence $f$ and $g$ are contiguous if and only if there is an element $\eta$ of ${\rm Map}(\K,\L)$ such that $f \leq \eta$ and $g \leq \eta$. Thus our definition of strong homotopy coincides with the original one of \cite{BM}.

Let $\K$ and $\L$ be $\Gamma$-simplicial complexes. A simplicial multi-map $\eta \in \Map(\K, \L)$ is {\it $\Gamma$-equivariant} if $\gamma (\eta(v)) = \eta(\gamma v)$ for each $v \in V(\K)$ and $\gamma \in \Gamma$. The induced subposet of $\Map(\K,\L)$ consisting of the $\Gamma$-equivariant multi-maps is denoted by $\Map_\Gamma(\K , \L)$. Two $\Gamma$-simplicial maps are strongly $\Gamma$-homotopic if they belong to the same connected component of $\Map_\Gamma(\K,\L)$.

\subsection{Posets}
For a poset $P$, the {\it order complex $\Delta(P)$ of $P$} is the abstract simplicial complex whose vertex set is the underlying set of $P$ and whose simplices are the finite chains of $P$. The {\it classifying space of $P$} is the geometric realization of $\Delta(P)$, and is denoted by $|P|$. It is easy to see that this definition coincides with the usual definition of the classifying space, i.e. the geometric realization of the nerve of $P$.

Let $f$ and $g$ be order-preserving maps from $P$ to $Q$. We write $f \leq g$ to mean that $f(x) \leq g(x)$ for every element $x$ of $P$. The poset of order-preserving maps from $P$ to $Q$ is denoted by ${\rm Poset}(P,Q)$. The order-preserving maps $f$ and $g$ are {\it strongly homotopic} if they belong to the same connected component of ${\rm Poset}(P,Q)$. If $f$ and $g$ are strongly homotopic, then they are homotopic, i.e. $|f|, |g| : |P| \rightarrow |Q|$ are homotopic. In fact $f$ and $g$ induce simplicially homotopic maps between the nerves (see Proposition 14.2.10 of \cite{Hirschhorn}), and many results concerning strong homotopy theory in this section and Lemma \ref{lem 2.2.1} and Corollary \ref{cor 2.2.2} hold for simplicial homotopy.

We call a $\Gamma$-equivariant order-preserving map between $\Gamma$-posets a $\Gamma$-poset map, for short. If $P$ and $Q$ are $\Gamma$-posets, then we denote by ${\rm Poset}_\Gamma (P,Q)$ the induced subposet of ${\rm Poset}(P,Q)$ consisting of $\Gamma$-poset maps. Two $\Gamma$-poset maps are strongly $\Gamma$-homotopic if they belong to the same connected component of ${\rm Poset}_\Gamma(P,Q)$. It is easy to see that if $f$ and $g$ are strongly $\Gamma$-homotopic, then they are also $\Gamma$-homotopic.

The {\it face poset of a simplicial complex $\K$} is the poset of non-empty simplices of $\K$ ordered by inclusion, and is denoted by $F \K$. For a simplicial map $f : \K \rightarrow \L$, define the order-preserving map $Ff : F\K \rightarrow F\L$ by the correspondence $\sigma \mapsto f(\sigma)$.

\begin{lem}[Barmak-Minian \cite{BM}]\label{lem 2.3.2}
Let $\K$ and $\L$ be $\Gamma$-simplicial complexes, $f$ and $g$ $\Gamma$-simplicial maps from $\K$ to $\L$, and suppose that $f$ and $g$ are strongly $\Gamma$-homotopic. Then the order-preserving maps $Ff$ and $Fg$ are strongly $\Gamma$-homotopic.
\end{lem}
\begin{proof}
This lemma clearly follows from the fact that the map
$$F :\Map_\Gamma (\K,\L) \longrightarrow {\rm Poset}_\Gamma (F\K,F\L), \; \eta \mapsto \Big(\sigma \mapsto \bigcup_{v \in \sigma} \eta(v) \Big)$$
is an order-preserving map.
\end{proof}

Thus if two simplicial maps $f,g : \K \rightarrow \L$ are strongly $\Gamma$-homotopic, then they are actually $\Gamma$-homotopic, i.e. $|f|$ and $|g|$ are $\Gamma$-homotopic.

\subsection{Strong collapse}

In this section, we consider the notion of the deformation retracts in the sense of the strong homotopy theory of $\Gamma$-posets and $\Gamma$-simplicial complexes. The goal of this section is Corollary \ref{cor 2.4.5}. Although the precise statements of many of the results given here did not appear in Barmak and Minian \cite{BM}, the ideas of the proofs already appearing there. However, for the reader's convenience, we shall give precise proofs.

\begin{dfn}
Let $P$ be a $\Gamma$-poset and $Q$ an induced $\Gamma$-subposet of $P$. Define ${\rm Def}_\Gamma(P,Q)$ to be the induced subposet of ${\rm Poset}_\Gamma (P,P)$ consisting of the $\Gamma$-poset maps which fix each point of $Q$. {\it $P$ strongly $\Gamma$-collapses to $Q$} if there is an element $f$ belonging to the identity component of ${\rm Def}_\Gamma(P,Q)$ with image contained in $Q$.
\end{dfn}

\begin{eg} \label{eg 2.4.0.0}
Let $c : P \rightarrow P$ be a $\Gamma$-equivariant closure operator, i.e. $c^2 = c$ and either $c \geq {\rm id}$ or $c \leq {\rm id}$ holds. Then $P$ strongly $\Gamma$-collapses to $c(P)$.
\end{eg}

The associated notion of simplicial complexes is similarly defined:

\begin{dfn}
Let $\L$ be an induced $\Gamma$-subcomplex of a $\Gamma$-simplicial complex $\K$. Let ${\rm Def}_\Gamma (\K,\L)$ be the induced subposet of $\Map_\Gamma (\K,\K)$ consisting of the $\Gamma$-equivariant simplicial multi-maps $\eta$ such that $\eta (v) = \{ v\}$ for every $v \in V(\L)$. {\it $\K$ strongly collapses to $\L$} if there is a simplicial map $f$ belonging to the identity component of ${\rm Def}(\K,\L)$ with image contained in $\L$.
\end{dfn}

\begin{rem} \label{rem 2.4.0.0}
Let $P$ be a poset and $Q$ is an induced subposet of $P$. Then $P$ strongly collapses to $Q$ if and only if there is a sequence $(f_0,\cdots, f_n)$ with finite length which satisfies the following conditions:
\begin{itemize}
\item[(1)] $f_0 = {\rm id}_P$ and $f_n(P) \subset Q$.
\item[(2)] $f_i(y) = y$ for every $y \in Q$ and $i=0,1,\cdots, n$.
\item[(3)] $f_i$ and $f_{i-1}$ are comparable for $i = 1,\cdots, n$.
\end{itemize}
Replacing ``comparable" to ``contiguous", we have a similar formulation for simplicial complexes.
\end{rem}

Recall that the group $\Gamma$ is assumed to be finite. Here we note the following obvious lemma, whose proof is omitted.

\begin{lem} \label{lem 3.5.1}
Let $\Gamma$ be a finite group, and $x$ an element of a $\Gamma$-poset $P$. If $\gamma x$ and $x$ are comparable, then we have $\gamma x = x$.
\end{lem}

Let $P$ be a $\Gamma$-poset. Recall that a point $x \in P$ is an {\it upper beat point} if $P_{>x} = \{ y \in P \; | \; y > x\}$ has the minimum, and $x$ is a {\it lower beat point of $P$} if $P_{<x} = \{ y \in P \; | \; y < x\}$ has the maximum. A point $x$ is a {\it beat point} if $x$ is either an upper or lower beat point of $P$. If $x$ is a beat point, then $P$ strongly collapses to $P \setminus \Gamma x$. In fact, if $x$ is an upper beat point and $y$ is the minimum of $P_{> x}$, then $y$ does not belong the orbit $\Gamma x$ of $x$ (Lemma \ref{lem 3.5.1}). Since the map
$$f: P \rightarrow P, \; f(z) = \begin{cases}
z & (z \not\in \Gamma x)\\
\gamma y & (z = \gamma x, \; \gamma \in \Gamma)
\end{cases}$$
is a closure operator and $f(P) = P \setminus \Gamma x$. Thus $P$ strongly $\Gamma$-collapses to $P \setminus \Gamma x$.

\begin{lem} \label{lem 2.4.0.1}
Let $Q \subset P' \subset P$ be a sequence of $\Gamma$-posets. Suppose that $P$ strongly $\Gamma$-collapses to $P'$. Then $P$ strongly $\Gamma$-collapses to $Q$ if and only if $P'$ strongly $\Gamma$-collapses to $Q$.
\end{lem}
\begin{proof}
Let $r: P \rightarrow P'$ be a retract of the inclusion $i: P' \hookrightarrow P$. Then the map
$${\rm Def}_\Gamma (P,Q) \rightarrow {\rm Def}_\Gamma (P',Q), \; f \mapsto  r \circ f \circ i$$
is an order-preserving map. This implies that $P'$ strongly $\Gamma$-collapses to $Q$ if $P$ strongly $\Gamma$-collapses to $Q$. It is easy to show the converse by Remark \ref{rem 2.4.0.0}.
\end{proof}

\begin{prop} \label{prop 2.4.1}
Let $Q$ be an induced $\Gamma$-subposet of a finite $\Gamma$-poset $P$. Then $P$ strongly $\Gamma$-collapses to $Q$ if and only if there is a linear order $\{ \alpha_1,\cdots, \alpha_n\}$ on $\Gamma \backslash (P \setminus Q)$ which satisfies the following: For each $i = 1,\cdots, n$, $\alpha_i$ is a family of beat points of $P \setminus ( \alpha_1 \cup \cdots \cup \alpha_{i-1})$ of $P$ (This order of $\Gamma \backslash (P \setminus Q)$ is independent from the order of $P$).
\end{prop}
\begin{proof}
By Lemma \ref{lem 2.4.0.1}, it suffices to show the following claim: Suppose that a finite poset $P$ strongly $\Gamma$-collapses to an induced subposet $Q$ of $P$ and $P \neq Q$. Then there is a beat point of $P$ not belonging to $Q$.

By the hypothesis, there is a map $f \in {\rm Def}_\Gamma (P,Q)$ such that $f < {\rm id}_P$ or $f > {\rm id}_P$. If $f >{\rm id}_P$, a maximal element of $\{ x \in P \; | \; f(x) > x\}$ is an upper beat point of $P$ not belonging to $Q$. The case $f < {\rm id}_P$ is similarly proved.
\end{proof}

Let $\K$ be a finite simplicial complex. A vertex $v$ of $\K$ is {\it dominated} (see Definition 2.1 of \cite{BM}) if there exists another vertex $w$ which satisfies the following condition: If a simplex $\sigma$ of $\K$ contains $v$, then $\sigma \cup \{ w\}$ is a simplex of $\K$. It is easy to see that if $v$ is dominated in $\K$, then $\K$ strongly collapses to $\K \setminus \Gamma v$. Here $\K \setminus \Gamma v$ is the subcomplex of $\K$ whose simplex is a simplex of $\K$ does not contain an element of $\Gamma v$. On the other hand, we have the following:

\begin{prop} \label{prop 2.4.3}
Let $P$ be a finite $\Gamma$-poset and $Q$ an induced $\Gamma$-subposet of $P$. If $P$ strongly $\Gamma$-collapses to $Q$, then $\Delta(P)$ strongly $\Gamma$-collapses to $\Delta(Q)$.
\end{prop}
\begin{proof}
If $x$ is a beat point of $P$, then $x$ is dominated in $\Delta (P)$. Thus this proposition follows from Proposition \ref{prop 2.4.1}.
\end{proof}

\begin{prop} \label{prop 2.4.4}
Let $\K$ be a $\Gamma$-simplicial complex and $\L$ an induced $\Gamma$-subcomplex of $\K$. Suppose that $\K$ strongly $\Gamma$-collapses to $\L$. Then $F\K$ strongly collapses to $F\L$.
\end{prop}
\begin{proof}
It suffices to note that the order-preserving map
$$F :\Map_\Gamma (\K,\L) \longrightarrow {\rm Poset}_\Gamma (F\K,F\L)$$
in the proof of Lemma \ref{lem 2.3.2} maps ${\rm Def}_\Gamma(\K,\L)$ to ${\rm Def}_\Gamma (F\K,F\L)$.
\end{proof}

Let $\K$ be a simplicial complex. The simplicial complex $\Delta(F\K)$ is called the {\it barycentric subdivision of $\K$}, and is denoted by $\Sd(\K)$.

\begin{cor} \label{cor 2.4.5}
Let $\K$ be a finite $\Gamma$-simplicial complex and $\L$ a $\Gamma$-subcomplex of $\K$. If $\K$ strongly $\Gamma$-collapses to $\L$, then $\Sd(\K)$ strongly $\Gamma$-collapses to $\Sd(\L)$.
\end{cor}



\subsection{$\times$-homotopy deformation retract}

In this section, we introduce the $\times$-homotopy deformation retract of graphs and prove Lemma \ref{lem 3.4.7}.

Let $G$ be a graph and $H$ an induced subgraph of $G$. Define the poset ${\rm Def}(G,H)$ to be the induced subposet of $\Hom(G,G)$ consisting of the multi-homomorphisms $\eta$ such that $\eta(w) = \{ w\}$ for every vertex $w$ of $H$. $H$ is a {\it $\times$-homotopy deformation retract of $G$} if there is a graph homomorphism $f$ belonging to the identity component of ${\rm Def}(G,H)$ such that $f(v) \in V(H)$ for every $v \in V(G)$.

Let $G$ and $H$ be left $\Gamma$-graphs. A multi-homomorphism $\eta$ from $G$ to $H$ is $\Gamma$-equivariant if $\gamma (\eta(v)) = \eta(\gamma v)$ for every $v \in V(G)$ and $\gamma \in \Gamma$. We let $\Hom_{\Gamma}(G,H)$ be the induced subposet of $\Hom(G,H)$ consisting of $\Gamma$-equivariant multi-homomorphisms. Then we have a poset map
$$Q :\Hom_{\Gamma}(G,H) \longrightarrow \Hom(\Gamma \backslash G, \Gamma \backslash H)$$
defined as follows: For a vertex $v \in V(G)$, $Q(\eta)(\Gamma v) = \{ \Gamma w \; | \; w \in \eta(v)\}$. Then we have the following proposition. For the definitions of $\A (\K)$ and $\A_T(\K)$, see the end of Section 3.1.

\begin{prop} \label{prop 4.3.1}
Let $(\K,\L)$ be a pair of $\Gamma$-simplicial complexes such that $\L$ is $\K$ strongly $\Gamma$-collapses to $\L$, and let $T$ be a right $\Gamma$-graph. Then $\A_T(\L)$ is a $\times$-homotopy deformation retract of $\A_T(\K)$.
\end{prop}
\begin{proof}
Consider the map
$$\Phi : \Map_{\Gamma} (\K,\K) \longrightarrow \Hom_{\Gamma}(\A(\K), \A(\K)), \; \eta \longmapsto (v \mapsto \eta(v)).$$
It is easy to see that $\Phi$ is well-defined and order-preserving. Next define the map $T \times ( - ) : \Hom_{\Gamma}(\A(\K),\A(\K)) \rightarrow \Hom_{\Gamma}(T \times \A(\K) , T \times \A(\K))$ by
$$(T \times \eta)(x,v) = \{ x\} \times \eta (v).$$
Thus we have a sequence
$$\begin{CD}
{\rm Map}_{\Gamma}(\K,\K) @>{\Phi}>> \Hom_{\Gamma}(\A(\K), \A(\K))\\
@>{T \times (-)}>> \Hom_\Gamma (T \times \A(\K), T \times \A(\K))\\
@>Q>> \Hom(\A_T(\K), \A_T(\K))
\end{CD}$$
of order-preserving maps, where the last map $Q$ is described in the previous paragraph of this proposition. The composition of the above sequence maps ${\rm Def}_{\Gamma}(\K,\L)$ to ${\rm Def}(\A_T(\K), \A_T(\L))$. Thus the proposition follows.
\end{proof}

\noindent {\it Proof of Lemma \ref{lem 3.4.7}}.
Note that
$$\hat{A}_T ((\Gamma / \Gamma') \times \Lambda_r[n]) = A_T \circ {\rm Sd}^k((\Gamma / \Gamma') \times \Lambda_r [n]) \cong \A_T \circ {\Sd}^k ((\Gamma / \Gamma') \times \Lam_r^n).$$
Since any vertex of $(\Gamma / \Gamma') \times \mathsf{\Lambda}^r_n$ other than elements of $(\Gamma / \Gamma') \times \{ r\}$ is dominated, $(\Gamma / \Gamma') \times \Lam_r^n$ strongly $\Gamma$-collapses to $\Gamma / \Gamma'$. Thus $\A_T \circ \Sd^k (\Gamma / \Gamma')$ is a $\times$-deformation retract of $\A_T \circ \Sd^k((\Gamma / \Gamma') \times \Lam_r^n)$ (Proposition \ref{prop 4.3.1}). The proof of (2) is similar.

We now show (3). Consider the diagram
$$\begin{CD}
(\Gamma /\Gamma') @>{\simeq_{\Gamma}}>> (\Gamma / \Gamma') \times \Lambda_r[n]\\
@VVV @VVV\\
\hat{S}_T \circ \hat{A}_T(\Gamma / \Gamma') @>>> \hat{S}_T \circ \hat{A}_T ((\Gamma / \Gamma') \times \Lambda_r[n]),
\end{CD}$$
where the vertical arrows are the unit maps. Since $T$ satisfies the condition (A), the left vertical arrow is a $\Gamma$-weak equivalence. By Lemma \ref{lem 2.2.1} and (1) of this proposition, the lower horizontal arrow is a $\Gamma$-weak equivalence. The proof of (4) is similar.
\qed

\vspace{2mm} We need the following assertion later:

\begin{lem} \label{lem x-pushout}
Let $H$ be a $\times$-deformation retract of a graph $G$, and $f : H \rightarrow Y$ a graph homomorphism. Let $X$ be the pushout $G \cup_f Y$. If $H$ is a $\times$-homotopy deformation retract of $G$, then $Y$ is a $\times$-deformation retract of $X$.
\end{lem}
\begin{proof}
Let $u : G \rightarrow X$ be the natural map. Define $\Phi : {\rm Def}(G,H) \rightarrow {\rm Def}(X,Y)$ by
$$\Phi(\eta)(v) = \begin{cases}
\{ v\} & (v \in V(Y)) \; \\
u(\eta (v)) & (v \in V(G)).
\end{cases}$$
It is easy to see that this map $\Phi$ is well-defined and order-preserving.
\end{proof}

\subsection{$r$-NDR of $\Gamma$-simplicial complex}

In this section, we introduce the notion of $r$-NDR's of $\Gamma$-simplicial complexes. A {\it pair of $\Gamma$-simplicial complexes} is a pair $(\K,\L)$ consisting of a $\Gamma$-simplicial complex $\K$ together with a $\Gamma$-subcomplex $\L$ of $\K$.

Let $\K$ be a simplicial complex. Recall that the star of $v \in V(\K)$ is the subcomplex of $\K$ defined by
$${\rm st}_{\K}(v) = \{ \sigma \in \K \; | \; \sigma \cup \{ v \} \in \K\}.$$

\begin{dfn}
Let $\K$ be a finite $\Gamma$-simplicial complex and $\L$ a $\Gamma$-subcomplex of $\K$. The {\it neighborhood of $\L$ in $\K$} is the $\Gamma$-subcomplex
$$\nu_{\K}(\L) = \nu(\L) = \bigcup_{v \in V(\L)} {\rm st}_{\K}(v) \subset \K.$$
For a positive integer $r$, define the {\it $r$-neighborhood $\nu^r(\L)$} inductively by $\nu^1(\L) = \nu(\L)$, and $\nu^{s+1}(\L) = \nu(\nu^s(\L))$.

A pair of $\Gamma$-simplicial complexes $(\K,\L)$ is an {\it $r$-NDR pair} if there exists a $\Gamma$-subcomplex $\A$ of $\K$ containing $\nu^r(\L)$ such that $\A$ strongly $\Gamma$-collapses to $\L$.
\end{dfn}

\begin{prop} \label{prop 2.4.6}
If $(\K,\L)$ is an $r$-NDR pair of finite $\Gamma$-simplicial complexes, then the pair $(\Sd(\K), \Sd(\L))$ is a $(2r)$-NDR pair.
\end{prop}

To prove Proposition \ref{prop 2.4.6}, we use the following lemma:

\begin{lem} \label{lem 2.4.7}
Let $(\K,\L)$ be a pair of simplicial complexes. Then we have
$$\nu^2(\Sd(\L)) \subset \Sd(\nu(\L)).$$
\end{lem}
\begin{proof}
Recall that a simplex of $\Sd (\K) = \Delta(F\K)$ is a chain of the face poset of $\K$. Let $c$ be a simplex of $\nu^2(\Sd(\L)) \subset \Sd(\K)$. By the definition of $\nu^2$, there is a vertex $\sigma$ of $\nu(\Sd(\L))$ such that $c \in {\rm st}_{\Sd (\K)}(\sigma)$, namely, $c \cup \{ \sigma\}$ is a chain of $F\K$. Since $\sigma \in V(\nu (\Sd(\L)))$, there is $\tau \in V(\Sd(\L))$ such that $\{\sigma \} \in {\rm st}_{{\rm Sd}(K)}(\tau)$, namely, $\{ \sigma, \tau\}$ is a chain in $F \K$.
Then the maximum $\sigma'$ of $c \cup \{ \sigma \}$ contains some element $v$ of $\tau$, and hence we have $\sigma' \in {\rm st}_{\K}(v) \subset \nu(\L)$. Therefore every element of $c \cup \{ \sigma\}$ belongs to $\nu(\L)$, and hence $c \subset c \cup \{ \sigma\} \in \Sd(\nu(\L))$.
\end{proof}

\noindent {\it Proof of Proposirion \ref{prop 2.4.6}.} Suppose that $(\K,\L)$ is an $r$-NDR pair of $\Gamma$-simplicial complexes, and let $\A$ be a $\Gamma$-subcomplex of $\K$ containing $\nu^r(\L)$ such that $\A$ strongly $\Gamma$-collapses to $\L$. By Lemma \ref{lem 2.4.7}, we have
$$\nu^{2r}(\Sd(\L)) \subset \Sd(\nu^r(\L)) \subset \Sd(\A).$$
Corollary \ref{cor 2.4.5} implies that $\Sd(\A)$ strongly $\Gamma$-collapses to $\Sd(\L)$. Therefore the pair $(\Sd(\K), \Sd(\L))$ is a $(2r)$-NDR pair.
\qed

\begin{thm} \label{thm 2.4.8}
Let $\K$ be a finite $\Gamma$-simplicial complex and $\L$ a $\Gamma$-subcomplex of $\K$. Then the pair $(\Sd^2(\K), \Sd^2(\L))$ is a 1-NDR pair.
\end{thm}
\begin{proof}
Note that $\Sd(\K)$ is the $\Gamma$-simplicial complex whose simplex is a finite chain of $\K \setminus \{ \emptyset\}$ with respect to the inclusion ordering. Set
$$X = \K \setminus \L = \{ \sigma \in \K \; | \; \sigma \not\in \L\}.$$
Then we have $F \Sd(\L) = \{ c \in F \Sd(\K) \; | \; c \cap X = \emptyset\}$. Set
\begin{eqnarray*}
P & = & \{ c \in F \Sd (\K ) \; | \; \textrm{There exists $\sigma \in \L$ with $\sigma \in c$.}\}\\
& = & \{ c \in F \Sd (\K ) \; | \; c \not\subset X\}.
\end{eqnarray*}
Note that $\Delta (P)$ is the 1-neighborhood of $\Sd^2(\L)$ in $\Sd^2(\K)$. Thus it suffices to show that $P$ strongly $\Gamma$-collapses to $F \Sd(\L)$ (Proposition \ref{prop 2.4.3}).

Define the closure operator (see Example \ref{eg 2.4.0.0}) $f: P \rightarrow P$ by $f(c) = c \cap \L$. Then $f(P) = F \Sd(\L)$. Thus the theorem follows.
\end{proof}

Combining Proposition \ref{prop 2.4.6} and Theorem \ref{thm 2.4.8}, we have the following:

\begin{cor} \label{cor 2.4.9}
Let $(\K,\L)$ be a pair of finite $\Gamma$-simplicial complexes. Then for $r \geq 2$, the pair $(\Sd^r(\K), \Sd^r(\L))$ is a $2^{r-2}$-NDR.
\end{cor}

\subsection{$r$-NDR for graphs}

In this section, we introduce the $r$-NDR pair of graphs and prove Proposition \ref{prop 3.4.8}.

\begin{dfn}
Let $G$ be a graph and $H$ a subgraph of $G$. Let $\nu_G(H) = \nu(H)$ be the subgraph of $G$ defined by
$$V(\nu(H)) = \{ v \in V(G) \; | \; \textrm{There is $w \in V(H)$ such that $(v,w) \in E(G)$.}\},$$
$$E(\nu(H)) = \{ (v,w) \in E(G) \; | \; \textrm{One of $v$ and $w$ is a vertex of $H$.}\}.$$
For $r \geq 1$, define the {\it $r$-neighborhood $\nu^r_G(H) = \nu^r(H)$ of $H$} inductively by $\nu^1(H) = \nu(H)$ and $\nu^{s+1}(H) = \nu (\nu^s(H))$.
\end{dfn}

\begin{prop} \label{prop 4.3.2}
Let $(\K,\L)$ be a pair of $\Gamma$-simplicial complexes. Then
$$\nu_{\A_T(\K)}^r(\A_T(\L)) \subset \A_T(\nu_{\K}^r(\L)).$$
\end{prop}
\begin{proof} The case $r = 1$ is deduced from the construction of $\A_T$ (Section 3.1) and is omitted. Thus we have
$$\nu_{\A_T(\K)}^r (\A_T(\K)) \subset \nu_{\A_T(\K)}^{r-1} (\A_T(\nu_{\K}(\L))) \subset \cdots \subset \A_T(\nu_{\K}^r(\L)).\qedhere$$
\end{proof}

\begin{cor} \label{cor 4.3.3}
If $(\K,\L)$ is an $r$-NDR pair of $\Gamma$-simplicial complexes, then the pair of graphs $(\A_T(\K), \A_T(\L))$ is an $r$-NDR pair of graphs.
\end{cor}
\begin{proof}
Let $\L'$ be a subcomplex of $\K$ containing $\nu^r_{\K}(\L)$ such that $\L'$ strongly $\Gamma$-collapses to $\L$. By Proposition \ref{prop 4.3.2}, we have
$$\nu^r_{\A_T(\K)}(\A_T(\L)) \subset \A_T(\nu^r_{\K}(\L)) \subset \A_T(\L').$$
Proposition \ref{prop 4.3.1} implies that $\A_T(\L)$ is a $\times$-deformation retract of $\A_T(\L')$.
\end{proof}

\begin{cor} \label{cor 4.3.4}
Let $(\K,\L)$ be a pair of finite $\Gamma$-simplicial complexes. Then for $r \geq 2$, the pair $(\A_T(\Sd^r(\K)), \A_T(\Sd^r(\L)))$ is a $2^{r-2}$-NDR pair of graphs.
\end{cor}
\begin{proof}
This follows from Corollary \ref{cor 2.4.5}, Proposition \ref{prop 4.3.1}, and Corollary \ref{cor 4.3.3}.
\end{proof}

Let $T$ be a finite connected right $\Gamma$-graph. The following theorem asserts that if $r$ is sufficiently large, then the class of $r$-NDR's satisfies the gluing lemma with respect to $\Sing_T$-complexes.

\begin{thm} \label{thm 4.3.5}
Let $r$ be a positive integer, $(G,H)$ an $r$-NDR pair of graphs, $f: H \rightarrow Y$ a graph homomorphism, and $X$ the pushout $Y \cup_H G$. Suppose that the finite right $\Gamma$-graph $T$ has at least one edge and the diameter of $T$ is smaller than $r$. Then the diagram
$$\begin{CD}
\Sing_T(H) @>>> \Sing_T(G)\\
@VVV @VVV\\
\Sing_T(Y) @>>> \Sing_T(X)
\end{CD}$$
is a homotopy pushout square in the category ${\bf SSet}^{\Gamma}$ of $\Gamma$-simplicial sets. In other words, the natural map
$$|\Sing_T(Y)| \cup_{|\Sing_T(H)|} |\Sing_T(G)| \longrightarrow |\Sing_T(X)|$$
is a $\Gamma$-homotopy equivalence. A similar assertion holds for $\Hom_T$-complexes.
\end{thm}
\begin{proof}
By Proposition \ref{Cube lemma} and Corollary \ref{cor 4.2.2}, the case of $\Hom_T$-complexes follows from the case of $\Sing_T$-complexes. Thus we only give the proof of the case of $\Sing_T$. Since $(G,H)$ is an $r$-NDR pair, there is a subgraph $H'$ of $G$ containing $\nu^r(H)$ such that $H$ is a $\times$-deformation retract of $H'$. Let $Y'$ be the pushout $Y \cup_H H'$. Then $Y$ is a $\times$-deformation retract of $Y'$.

Consider the commutative square
$$\begin{CD}
\Sing_T(H') @>{i_*}>> \Sing_T(G)\\
@V{f'_*}VV @VV{\hat{f}_*}V\\
\Sing_T(Y') @>{j_*}>> \Sing_T(X),
\end{CD}$$
where $\hat{f} : G \rightarrow X = Y\cup_H G$ and $f' : H' \rightarrow Y' = Y \cup_H H' $ are the natural maps, $i : H' \hookrightarrow G$ and $j : Y' \hookrightarrow X$ are inclusions. We claim that the above square is a pushout square. To see this, we want to show that the diagram

\begin{eqnarray} \label{eqn 4.3.1}
\begin{CD}
\mathcal{G}(T \times \Sigma^n, H') @>{i_*}>> \mathcal{G}(T \times \Sigma^n, G)\\
@V{f'_*}VV @VV{\hat{f}_*}V\\
\mathcal{G}(T \times \Sigma^n, Y') @>{j_*}>> \mathcal{G}(T \times \Sigma^n, X)
\end{CD}
\end{eqnarray}
is a pushout diagram in the category ${\bf Set}^\Gamma$ of $\Gamma$-sets. Let $\varphi : T \times \Sigma^n \rightarrow X$ be a graph homomorphism. If the image of $\varphi$ does not intersect $Y$, then $\varphi$ factors through $G$. Suppose that there is a vertex $v$ of $T \times \Sigma^n$ with $\varphi (v) \in V(Y)$. Since the diameter of $T \times \Sigma^n$ is smaller than or equal to $r$ and $Y'$ is the $r$-neighborhood of $Y$ in $X$, we have that $\varphi$ factors through $Y'$. Thus we have shown that the map
$$\hat{f}_* \sqcup j_* : \mathcal{G}(T \times \Sigma^n , G) \sqcup \mathcal{G}(T \times \Sigma^n, Y') \longrightarrow \mathcal{G}(T \times \Sigma^n, X)$$
is surjective. Next let $\psi_0 : T \times \Sigma^n \rightarrow G$ and $\psi_1 : T \times \Sigma^n \rightarrow Y'$ be graph homomorphisms with $\hat{f} \psi_0 = j \psi_1$. Since the image of $\hat{f} \psi_0$ is contained in $Y'$, we have that $\psi_0$ factors through $H'$, and let $\psi : T \times \Sigma^n \rightarrow H'$ with $i \psi = \psi_0$. Since
$$j f' \psi = \hat{f} i \psi = \hat{f} \psi_0 = j \psi_1$$
and $j$ is a monomorphism, we have that $f' \psi = \psi_1$. Thus the diagram (\ref{eqn 4.3.1}) is a pushout diagram.

Next we consider the commutative diagram
\begin{eqnarray} \label{eqn 4.3.2}
\begin{CD}
\Sing_T(Y) @<<< \Sing_T(H) @>>> \Sing_T(G)\\
@V{j_*}VV @V{i_*}VV @|\\
\Sing_T(Y') @<<< \Sing_T(H') @>>> \Sing_T(G).
\end{CD}
\end{eqnarray}
Since $i$ and $j$ are $\times$-homotopy equivalences, we have that the all vertical arrows in the above diagram are $\Gamma$-weak equivalences (Lemma \ref{lem 2.2.1} and Corollary \ref{cor 4.2.2}). Let $E$ (or $E'$) be the homotopy pushout (see Chapter 13 of \cite{Hirschhorn}) of the upper (or lower, respectively) horizontal arrows. Then we have a commutative diagram
$$\begin{CD}
E @>{\simeq_\Gamma}>> \Sing_T(Y) \cup_{\Sing_T(H)} \Sing_T(G)\\
@V{\simeq_\Gamma}VV @VVV\\
E' @>{\simeq_\Gamma}>> \Sing_T(X).
\end{CD}$$
The left vertical arrow is a $\Gamma$-weak equivalence since the all vertical arrows in commutative diagram (\ref{eqn 4.3.2}) are $\Gamma$-weak equivalences (see Proposition 13.5.3 of \cite{Hirschhorn}). The horizontal arrows are $\Gamma$-weak equivalences since the map $\Sing_T(H) \rightarrow \Sing_T(G)$ and $\Sing_T(H') \rightarrow \Sing_T(G)$ are cofibrations in ${\bf SSet}^\Gamma$, and
$$\Sing_T(X) = \Sing_T(Y') \cup_{\Sing_T(H')} \Sing_T(G)$$
as was proved (see Corollary 13.3.8 of \cite{Hirschhorn}).
\end{proof}

\noindent {\it Proof of Proposition \ref{prop 3.4.8}}.
By Proposition \ref{Cube lemma} and the hypothesis, the map 
\begin{eqnarray} \label{unit pushout}
u_K \cup_{u_L} u_{L'} : K' = K \cup_L L' \rightarrow \hat{S}_T \circ \hat{A}_T(K) \cup_{\hat{S}_T \circ \hat{A}_T (L)} \hat{S}_T \circ \hat{A}_T(L')
\end{eqnarray}
is a $\Gamma$-weak equivalence. Since $(\hat{A}_T(K), \hat{A}_T(L))$ is a $2^{k-2}$-NDR pair (Corollary \ref{cor 4.3.4}), Theorem \ref{thm 4.3.5} implies that the map
$$S_T \circ \hat{A}_T(K) \cup_{S_T \circ \hat{A}_T(L)} S_T \circ \hat{A}_T(L') \rightarrow S_T \circ \hat{A}_T(K')$$
is a $\Gamma$-weak equivalence. Here we write $S_T$ instead of $\Sing_T$. Consider the commutative diagram
$$\begin{CD}
S_T \circ \hat{A}_T(K) \cup_{S_T \circ \hat{A}_T(L)} S_T \circ \hat{A}_T(L') @>{\simeq_\Gamma}>> S_T \circ \hat{A}_T(K')\\
@VVV @VV{\simeq_\Gamma}V\\
\hat{S}_T \circ \hat{A}_T(K) \cup_{\hat{S}_T \circ \hat{A}_T(L)} \hat{S}_T \circ \hat{A}_T (L') @>{u'}>> \hat{S}_T \circ \hat{A}_T(K').
\end{CD}$$
Proposition \ref{Cube lemma} implies that the left vertical arrow is a $\Gamma$-weak equivalence. Thus the lower horizontal arrow $u'$ is a $\Gamma$-weak equivalence. Since (\ref{unit pushout}) is a $\Gamma$-weak equivalence, we have that $u_{K'} = u' \circ (u_K \cup_{u_L} u_{L'})$ is a $\Gamma$-weak equivalence.
\qed

\end{document}